\documentclass[a4paper]{amsart}

\usepackage[french]{babel}
\usepackage[utf8]{inputenc}
\usepackage[T1]{fontenc}
\usepackage{lmodern} 

\usepackage{enumerate}

\newcommand{\N}{\mathbb{N}}
\newcommand{\Z}{\mathbb{Z}}
\newcommand{\Q}{\mathbb{Q}}
\newcommand{\R}{\mathbb{R}}

\newtheorem{theoreme}{Th\'eor\`eme}
\newtheorem{proposition}{Proposition}
\newtheorem{lemme}{Lemme}

\newtheorem{remarque}{Remarque}

\usepackage{geometry}
\usepackage{url}
\usepackage{hyperref}

\begin{document}

\title{Preuve simplifiée du théorème de Serret sur les nombres équivalents}

\author{A. Bauval}\thanks{Anne Bauval, bauval@math.univ-toulouse.fr, IMT, UMR 5219, Universit\'e Toulouse~III}


\begin{abstract}

\noindent
\emph{Si les fractions continues de deux irrationnels $x$ et $y$ ont un quotient complet commun alors $x$ et $y$ sont équivalents, c'est-à-dire qu'il existe $a,b,c,d\in\Z$ tels que $ad-bc=\pm1$ et $y=\frac{ax+b}{cx+d}$. La réciproque, due à Serret, est désormais classique, mais on en donne une preuve plus simple et purement algébrique.}
\end{abstract}

\maketitle

\section{Préliminaires}\label{sec:Prelim}

Un irrationnel $x$, décrit par sa fraction continue
$$x=[a_0,a_1,\dots]=a_0+\cfrac1{a_1+\cfrac{1}{a_2+\dots}}\quad(\text{avec}\quad a_0\in\Z\quad\text{et}\quad a_n\in\N^*\quad\text{si}\quad n\ge1),$$
la suite $(a_n)$ de ses quotients partiels, et la suite $(x_n)$ de ses quotients complets, sont liés par :
$$\forall n\in\N\quad x=[a_0,\dots,a_{n-1},x_n]\quad\text{et}\quad x_n=[a_n,a_{n+1},\dots].$$
De plus, si $(p_n)$ et $(q_n)$ sont les deux suites d'entiers définies par
$$\begin{array}{ccc}p_{-2}=0,&p_{-1}=1,& p_n=a_np_{n-1}+p_{n-2},\\ q_{-2}=1,&q_{-1}=0,&q_n=a_nq_{n-1}+q_{n-2},\end{array}$$
pour tout $n\in\N$, on a
$$[a_0,\dots,a_n]=\frac{p_n}{q_n},\quad x=\frac{p_{n-1}x_n+p_{n-2}}{q_{n-1}x_n+q_{n-2}}\quad\text{et}\quad p_{n-1}q_{n-2}-p_{n-2}q_{n-1}=(-1)^n.$$
En particulier, $x$ est équivalent à tous les $x_n$, c'est-à-dire dans la même orbite pour l'action du groupe $G:=PGL_2(\Z)$ sur $\R\setminus\Q$ par homographies, ce que l'on notera : $x\in Gx_n$.

On notera d'autre part $\sim$ la relation d'équivalence sur les irrationnels ``avoir un quotient complet commun''. Autrement dit, pour tous irrationnels $x=[a_0,a_1,\dots]$ et $y=[b_0,b_1,\dots]$ :
$$\begin{array}{ccccc}
x\sim y&\Leftrightarrow&\exists i,j\in\N&&x_i=y_j\\
&\Leftrightarrow&\exists i,j\in\N&\forall k\in\N&a_{i+k}=b_{j+k}\\
&\Leftrightarrow&\exists i,j\in\N&\forall k\in\N&x_{i+k}=y_{j+k}.
\end{array}$$

Ces deux relations d'équivalence sur $\R\setminus\Q$ sont en fait la même :
\begin{theoreme}\cite{Serret} Deux irrationnels $x$ et $y$ sont équivalents si et seulement si $x\sim y$.
\end{theoreme}

Il est clair que $x\sim y\Rightarrow y\in Gx$. Nous proposons une preuve de la réciproque, inspirée de \cite{Lachaud} et plus simple que la démonstration originelle reproduite dans tous les manuels (\cite{Perron, HardyWright, Lang, RockettSzusz, Vorobiev, Borwein}, etc.).

\section{Démonstration}

\begin{lemme}\label{lem:Opp}Pour tout irrationnel $x$, on a $-x\sim x$.
\end{lemme}
\begin{proof}(\cite{Perron})
$$-[a_0,a_1,x_2]=\begin{cases}[-a_0-1,1,a_1-1,x_2]&\text{si }a_1>1\\
[-a_0-1,x_2+1]&\text{si }a_1=1.\end{cases}$$
\end{proof}

\begin{lemme}\label{lem:Inv}Pour tout irrationnel $x$, on a $1/x\sim x$.
\end{lemme}
\begin{proof}
$$1/[a_0,x_1]=\begin{cases}[0,a_0,x_1]&\text{si }a_0\ge1\\
x_1&\text{si }a_0=0\end{cases}$$
(\cite{Lang}), ce qui démontre le cas $x>0$. Le cas $x<0$ s'en déduit grâce au lemme précédent.
\end{proof}

\begin{proposition}\label{prop:Serret} Si deux irrationnels $x$ et $y$ sont équivalents alors $x\sim y$.
\end{proposition}
\begin{proof}
Soient $a,b,c,d\in\Z$ tels que $ad-bc=\pm1$ et $y=\frac{ax+b}{cx+d}$. Quitte à remplacer si nécessaire $(a,b,c,d)$ par son opposé (ce qui ne modifie pas $y$), on peut de plus supposer que $(c,d)=(0,1)$ ou $c>0$. Dans le premier cas, $y=\pm x+b\sim\pm x\sim x$ d'après le lemme \ref{lem:Opp}. Supposons maintenant $c>0$. Le rationnel $a/c$ admet deux développements en fraction continue finie, chacun se déduisant de l'autre en raccourcissant ou rallongeant artificiellement ce dernier de $1$. Choisissons celui,
$$\frac ac=[a_0,\dots,a_{n-1}]$$
(avec $n\ge1$), pour lequel la parité de $n$ est telle que $ad-bc=(-1)^n$. Les suites $(p_i), (q_i)_{i<n}$ étant définies comme dans la section \ref{sec:Prelim}, on a $\frac ac=\frac{p_{n-1}}{q_{n-1}}$ et même (puisque ces deux fractions sont irréductibles et de dénominateurs positifs)
$$a=p_{n-1}\quad\text{et}\quad c=q_{n-1}.$$
Puisque de plus $ad-bc=(-1)^n=p_{n-1}q_{n-2}-p_{n-2}q_{n-1}$, il existe un entier $r$ tel que
$$b=rp_{n-1}+p_{n-2}\quad\text{et}\quad d=rq_{n-1}+q_{n-2}$$
donc
$$y=\frac{ax+b}{cx+d}=\frac{p_{n-1}(x+r)+p_{n-2}}{q_{n-1}(x+r)+q_{n-2}}=[a_0,\dots,a_{n-1},x+r].$$
On conclut grâce au lemme \ref{lem:Inv} et à l'invariance de $\sim$ par translations entières :
$$y\sim[a_1,\dots,a_{n-1},x+r]\sim[a_2,\dots,a_{n-1},x+r]\sim\dots\sim x+r\sim x.$$
\end{proof}

\begin{remarque}{\rm On vient en fait de redémontrer le théorème suivant, classique\footnote{Plus précisément (\cite{Coxeter}) : $$PGL_2(\Z)=\langle T,U,V\mid U^2=V^2=(TV)^2=(UV)^2=(TUV)^3=1\rangle.$$} et un peu plus fort que la proposition \ref{prop:Serret} :}\\
\indent le groupe $PGL_2(\Z)$ est engendré par les trois éléments $T,U,V$\footnote{Ou même seulement par $T$ et $U$, puisque -- cf. preuve du lemme \ref{lem:Opp} -- $UV=T^{-1}UTUT^{-1}$. Plus précisément :
$$PGL_2(\Z)=\langle T,U\mid U^2=(UTUT^{-2})^2=(UTUT^{-1})^3=1\rangle.$$}
correspondant respectivement aux homographies
$$t:x\mapsto1+x,\quad u:x\mapsto1/x\quad\text{et}\quad v:x\mapsto -x,$$
{\rm tout en le précisant :}\\
\indent tout élément de $PGL_2(\Z)$ s'écrit sous la forme$$V^eT^{a_0}UT^{a_1}U\dots UT^{a_n}$$où tous les $a_i$ sont strictements positifs sauf éventuellement le premier et le dernier, et $e$ est égal à $0$, ou éventuellement à $1$ si $n=0$.\\
{\rm De plus, cette écriture est unique.}
\end{remarque}

\noindent\author{Anne Bauval,} bauval@math.univ-toulouse.fr\\
\address{\small Institut de Math\'ematiques de Toulouse\\
Universit\'e Toulouse III\\
118 Route de Narbonne, 31400 Toulouse - France\\
}

\end{document}